\documentclass[11pt, amsfonts, dvipdfmx]{amsart}

\usepackage{amsmath,amssymb,amsthm}
\usepackage[all]{xy}
\usepackage[dvipdfm,colorlinks=true]{hyperref}
\usepackage{tikz}

\numberwithin{equation}{section}

\SelectTips{eu}{12}

\newtheorem{theorem}{Theorem}[section]
\newtheorem{proposition}[theorem]{Proposition}
\newtheorem{lemma}[theorem]{Lemma}
\newtheorem{corollary}[theorem]{Corollary}

\theoremstyle{definition}
\newtheorem{definition}[theorem]{Definition}
\newtheorem{example}[theorem]{Example}

\theoremstyle{remark}

\newcommand{\Z}{\mathcal{Z}}

\newcommand{\C}{\mathbb{C}}

\newcommand{\F}{\mathcal{F}}
\newcommand{\RZ}{\mathbb{R}\mathcal{Z}}

\title{Whitehead products in moment-angle complexes}

\author{Kouyemon Iriye}
\address{Department of Mathematical Sciences, Osaka
Prefecture University, Sakai, 599-8531, Japan}
\email{kiriye@mi.s.osakafu-u.ac.jp}
\author{Daisuke Kishimoto}
\address{Department of Mathematics, Kyoto University, Kyoto, 606-8502, Japan}
\email{kishi@math.kyoto-u.ac.jp}


\begin{document}

\baselineskip.525cm

\maketitle

\begin{abstract}
In toric topology, to a simplicial complex $K$ with $m$ vertices, one associates two spaces, the moment-angle complex $\Z_K$ and the Davis-Januszkiewicz space $DJ_K$. These spaces are connected by a homotopy fibration $\Z_K\to DJ_K\to(\C P^\infty)^m$. In this paper, we show that the map $\Z_K\to DJ_K$ is identified with a wedge of iterated (higher) Whitehead products for a certain class of simplicial complexes $K$ including dual shellable complexes. We will prove the result in a more general setting of polyhedral products.
\end{abstract}


\section{Introduction}


\subsection{Moment-angle complex}

In a seminal work introducing quasitoric manifolds \cite{DJ}, Davis and Januszkiewicz constructed a space from a simple convex polytope (or equivalently, the dual simplicial convex polytope) as a topological analogue of the hyperplane arrangement appearing in the theory of toric varieties so that every quasitoric manifold is obtained as the quotient of the space by a certain torus action. Later on, the construction of this space is generalized to any simplicial complex as follows. Let $K$ be a simplicial complex on the vertex set $[m]=\{1,\ldots,m\}$. The moment-angle complex $\Z_K$ is defined as the union of subspaces $D(\sigma)=\{(z_1,\ldots,z_m)\in(D^2)^m\,\vert\,|z_i|=1\text{ for }i\not\in\sigma\}$ of $(D^2)^m$ for all $\sigma\in K$, where $D^2=\{z\in\C\,\vert\,|z|\le 1\}$.

The moment-angle complex is now a fundamental object not only as a source of quasitoric manifolds but also as a space connecting toric topology with a broad area of mathematics including algebraic geometry, algebraic topology, combinatorics, commutative algebra, and geometry. In particular, recent development of the study on the homotopy type of $\Z_K$ in connection with combinatorics and commutative algebra is significant \cite{GT3,GW,IK1,IK3}.


\subsection{Object of study}

Davis and Januszkiewicz \cite{DJ} also constructed a supplementary space from a simple convex polytope, and it is also generalized to any simplicial complex. The supplementary space associated with a simplicial complex $K$ is called the Davis-Januszkiewicz space and denoted by $DJ_K$ which is defined as the union of the subspaces $E(\sigma)=\{(x_1,\ldots,x_m)\in(\C P^\infty)^m\,\vert\,x_i=*\text{ for }i\not\in\sigma\}$ of $(\C P^\infty)^m$ for all $\sigma\in K$, where $*\in\C P^\infty$ is the basepoint.

By definition, there is a natural action of torus $(S^1)^m$ on $\Z_K$ and the Davis-Januszkiewicz space $DJ_K$ is homotopy equivalent to the Borel construction of this torus action. Then in particular, there is a homotopy fibration
\begin{equation}
\label{fibration_toric}
\Z_K\xrightarrow{\widetilde{w}}DJ_K\to(\C P^\infty)^m.
\end{equation}
The object of study in this paper is the fiber inclusion $\widetilde{w}$.


\subsection{Previous work and aim}

When $m=2$ and $K$ consists of two vertices, i.e. the boundary of 1-simplex, then we have $\Z_K=S^3$ and $DJ_K=\C P^\infty\vee\C P^\infty$ so that the homotopy fibration \eqref{fibration_toric} coincides with Ganea's homotopy fibration
$$S^3\to\C P^\infty\vee\C P^\infty\to(\C P^\infty)^2.$$
Thus in particular, the map $\widetilde{w}$ is the Whitehead product of the two copies of the bottom cell inclusion $S^2\to\C P^\infty$. More generally, if $K$ is the boundary of $(m-1)$-simplex, then the map $\widetilde{w}$ is the higher Whitehead product of the $m$-copies of the bottom cell inclusion $S^2\to\C P^\infty$.

This lets us study a connection between the map $\widetilde{w}$ and Whitehead products when $K$ is controlled by the boundaries of simplices, or equivalently, minimal non-faces which will be defined below. Along this line, Grbi\'c and Theriault \cite{GT2} introduced a directed MF-complex which is the union of well arranged boundaries of simplices and showed that $\Z_K$ is homotopy equivalent to a wedge of spheres, through which the map $\widetilde{w}$ is identified with a wedge of iterated Whitehead products including at most one higher Whitehead product. Their proof is complicated and seems to have a serious gap. They calculated the rational loop homology of $DJ_K$ and identified its generators with iterated Whitehead products. Then they deduce the result by observing the rational homotopy equivalences can be replaced with genuine homotopy equivalences, where there seems to be a gap in this step.

The aim of this paper is to generalize (and to give a correct proof of) the result of Grbi\'c and Theriault \cite{GT2} to a much larger class of simplicial complexes by a simpler method employing the recent result of the authors \cite{IK3} on the homotopy type of $\Z_K$, where the result of \cite{IK3} has applications \cite{IK4,IK5} on Golodness in combinatorial commutative algebra too. Abramyan \cite{A} showed that $\widetilde{w}$ is not necessarily a wedge of iterated Whitehead products even if $\Z_K$ is homotopy equivalent to a wedge of spheres. So the choice of $K$ is crucial.


\subsection{Polyhedral product}

In \cite{BBCG}, Bahri, Bendersky, Cohen, and Gitler unified and generalized the construction of $\Z_K$ and $DJ_K$, and introduced a space called a polyhedral product. Polyhedral products enable us to study the homotopy theory of $\Z_K$ and $DJ_K$ with a wider viewpoint and better homotopy theoretical techniques.

In our case, the map $\widetilde{w}$ can be defined in a more general setting using polyhedral products so that we will study this generalized map $\widetilde{w}$ in what follows, where the previous result \cite{GT2} actually considered polyhedral products. However, in this introduction, we will state our result in terms of $\Z_K$ and $DJ_K$ for readability.


\subsection{Totally fillable complex}

Now we introduce a simplicial complex for which we are going to study the map $\widetilde{w}$. We set notation. Let $L$ be a simplicial complex on the vertex set $V$. Let $|L|$ denote the geometric realization of $L$ and for a non-empty subset $I\subset V$, define the full subcomplex of $L$ on $I$ by $L_I=\{\sigma\in L\,\vert\,\sigma\subset I\}$. A subset $\sigma\subset V$ is called a minimal non-face of $L$ if it is not a simplex of $L$ and any of its proper subset is a simplex of $L$. In particular, if we add minimal non-faces to $L$, then we get a new simplicial complex.

\begin{definition}
A simplicial complex $K$ is called fillable if there is a collection of minimal non-faces $\{\sigma_1,\ldots,\sigma_r\}$ such that $|K\cup\sigma_1\cup\cdots\cup\sigma_r|$ is contractible. If any full subcomplex of $K$ is fillable, then it is called totally fillable.
\end{definition}

The collection of minimal non-faces $\{\sigma_1,\ldots,\sigma_r\}$ in the above definition is called a filling of $K$ and denoted by $\F(K)$, where there are possibly several fillings of a fillable complex $K$. The class of totally fillable complexes includes important simplicial complexes called dual shellable. We will show that directed MF-complexes that were considered in the previous work \cite{GT2} are dual shellable so that they are totally fillable too.


\subsection{Statement of the result}

The key to study the map $\widetilde{w}$  for a totally fillable complex $K$ is the following homotopy decomposition of $\Z_K$ which was obtained in \cite{IK3}.

\begin{theorem}
\label{decomp Z}
Let $K$ be a totally fillable complex on the vertex set $[m]$ with fillings $\F(K_I)$ of $K_I$ for all $\emptyset\ne I\subset[m]$. Then there is a homotopy equivalence
$$\Z_K\simeq\bigvee_{\emptyset\ne I\subset[m]}\bigvee_{\sigma\in\F(K_I)}S^{|\sigma|+|I|-1}.$$
\end{theorem}

Let $\widetilde{a}_i\colon S^2\to DJ_K$ be the inclusion of the bottom cell of the $i$-th $\C P^\infty$ in $DJ_K$. For $\sigma\subset[m]$ with $|\sigma|\ge 2$, let $\widetilde{w}_\sigma$ be the higher Whitehead product of $\widetilde{a}_i$ for $i\in\sigma$ if it is defined. Now we state our result.

\begin{theorem}
\label{main1}
Let $K$ be a totally fillable complex on the vertex set $[m]$ with fillings $\F(K_I)$ of $K_I$ for all $\emptyset\ne I\subset[m]$. The equivalence of Theorem \ref{decomp Z} can be chosen so that the composite
$$S^{|\sigma|+|I|-1}\to\bigvee_{\emptyset\ne I\subset[m]}\bigvee_{\sigma\in\F(K_I)}S^{|\sigma|+|I|-1}\simeq\Z_K\xrightarrow{\widetilde{w}}DJ_K$$
is the iterated Whitehead product
$$[\cdots[[\widetilde{w}_\sigma,\widetilde{a}_{i_1}],\widetilde{a}_{i_2}],\cdots,\widetilde{a}_{i_k}]$$
where $i_1,\ldots,i_k$ is a certain ordering of $I-\sigma$.
\end{theorem}

\emph{Acknowledgement:} The authors were partly supported by JSPS KAKENHI (No. 26400094 and No. 17K05248).


\section{Fillable complex}

Throughout this paper, let $K$ be a simplicial complex on the vertex set $[m]$. We will assume that a totally fillable complex $K$ is given specific fillings $\F(K_I)$ of $K_I$ for all $\emptyset\ne I\subset[m]$ unless otherwise is specified.


\subsection{Deletable complex}

In \cite{IK3}, it is proved that dual shellable complexes are totally fillable. The proof there actually shows that dual shellable complexes are in a certain subclass of totally fillable complexes. Here we introduce this subclass. A simplicial complex $K$ is called deletable if there are facets $\sigma_1,\ldots,\sigma_r$ such that $K-\{\sigma_1,\ldots,\sigma_r\}$ is collapsible, where $r$ can be 0, i.e. $K$ itself is collapsible. $K$ is called totally deletable if $\mathrm{lk}_{K_{[m]-S}}(v)$ is deletable for any $S\subset[m]$, possibly empty, and $v\in[m]-S$, where $\mathrm{lk}_L(w)=\{\sigma\in L\,\vert\,w\not\in\sigma,\,\sigma\cup w\in L\}$ is the link of a vertex $w$ of a simplicial complex $L$.

Let $L$ be a simplicial complex on the ground set $V$ which is a set including the vertex set. The Alexander dual of $L$ with respect to $V$, denoted $L^\vee$, is the simplicial complex consisting of $\sigma\subset V$ such that $V-\sigma$ is not a simplex of $L$. If we do not specify the ground set, then the Alexander dual will be taken over the vertex set. The following dictionary is useful, which is proved in \cite{IK3}. For a vertex $v$ of $L$, let $\mathrm{dl}_L(v)=\{\sigma\in L\,\vert\,\sigma\subset V-\{v\}\}$ be the deletion of $v$.

\begin{proposition}
  \label{dictionary}
  Let $L$ be a simplicial complex on the ground set $V$.
  \begin{enumerate}
    \item $(L^\vee)^\vee=L$, where the Alexander duals are taken over $V$.
    \item $\sigma\subset V$ is a facet of $L$ if and only if $\sigma^\vee$ is a minimal non-face of $L^\vee$, where $\sigma^\vee=V-\sigma$.
    \item For any $v\in V$, $\mathrm{dl}_L(v)^\vee=\mathrm{lk}_{L^\vee}(v)$, where the Alexander duals of $\mathrm{dl}_L(v)$ and $L$ are taken over $V-\{v\}$ and $V$, respectively.
  \end{enumerate}
\end{proposition}

The following is proved in \cite{IK3}.

\begin{proposition}
  \label{collapsible}
  If $K$ is collapsible, then $|K^\vee|$ is contractible.
\end{proposition}

Then by Proposition \ref{dictionary} one gets:

\begin{corollary}
  \label{deletable}
  Dual (totally) deletable complexes are (totally) fillable.
\end{corollary}


\subsection{Shellable complex}

Recall that $K$ is called shellable if there is an ordering of facets $\sigma_1,\ldots,\sigma_k$ of $K$, called a shelling, such that $\langle \sigma_1,\ldots,\sigma_{i-1}\rangle\cap\langle \sigma_i\rangle$ is pure and $(|\sigma_i|-2)$-dimensional for $i=2,\ldots,k$, where $\langle\tau_1,\ldots,\tau_r\rangle$ means the simplicial complex with facets $\tau_1,\ldots,\tau_r$. Shellable complexes were originally introduced as a combinatorial criterion for Cohen-Macaulayness and are now one of the most important class of simplicial complexes.

\begin{example}
\label{eg_shellable}
Any skeleton of a simplex is a shellable complex, where any ordering of its facets is a shelling.
\end{example}

As is seen in \cite{BW,IK3}, if $K$ is shellable, then $K$ is deletable and the link of any of its vertices is shellable. Then by Proposition \ref{dictionary} we get the following.

\begin{proposition}
  Shellable complexes are totally deletable.
\end{proposition}

By Corollary \ref{deletable}, we obtain:

\begin{corollary}
  Dual shellable complexes are totally fillable.
\end{corollary}

\begin{example}
\label{eg_dual_shellable}
A skeleton of a simplex is shellable as in Example \ref{eg_shellable}, and its Alexander dual is again a skeleton of a simplex which is obviously totally fillable.
\end{example}

\begin{example}
  Let $K$ be the following simplicial complex with six vertices.
  \begin{center}
  \begin{tikzpicture}[thick]
    \path[draw,fill][fill=gray!50!white,draw=black](0,0)--(1,0.7)--(0,1.4)--(0,0);
    \draw(1,0.7)--(1.7,0.7);
    \path[draw,fill][fill=gray!50!white,draw=black](1.7,0.7)--(2.7,1.4)--(2.7,0)--(1.7,0.7);
    \fill[black](0,0)circle(2pt);
    \fill[black](1,0.7)circle(2pt);
    \fill[black](0,1.4)circle(2pt);
    \fill[black](1.7,0.7)circle(2pt);
    \fill[black](2.7,1.4)circle(2pt);
    \fill[black](2.7,0)circle(2pt);
  \end{tikzpicture}
\end{center}
Then $K$ is collapsible, so it is deletable. Moreover, for any vertex $v$, $\mathrm{lk}_K(v)$ is either the interval graph or the disjoint union of the interval graph and one point. Then $\mathrm{lk}_K(v)$ is shellable for any vertex $v$, implying that $K$ is totally deletable. However, $K$ is not shellable obviously so that the class of dual totally deletable complexes is strictly larger than that of dual shellable complexes.
\end{example}



\subsection{Directed MF-complex}

In the previous work \cite{GT2}, Grbi\'c and Theriault introduced a simplicial complex called a directed MF-complex and studied the map $\widetilde{w}$ for a directed MF-complex $K$. A simplicial complex $K$ is called a directed MF-complex if for a collection of minimal non-faces $\{\sigma_1,\ldots,\sigma_r\}$, there is a filtration of subcomplexes $\emptyset=K_1\subset K_2\subset\cdots\subset K_r=K$ such that $K_i=K_{i-1}\cup\partial\sigma_i$ and $K_{i-1}\cap\partial\sigma_i$ is a face, where $\partial\sigma$ means the boundary of a simplex $\sigma$.

\begin{example}
\label{eg_totally_fillable}
The $k$-skeleton of an $n$-dimensional simplex is a directed MF-complex unless $k=n$.
\end{example}

\begin{proposition}
  Directed MF-complexes are dual shellable.
\end{proposition}

We shall show that directed MF-complexes are dual shellable. For this, we will use the following lemma.

\begin{lemma}
  \label{shelling}
  Suppose that there is an ordering of minimal non-faces $\sigma_1<\cdots<\sigma_r$ of $K$ such that for any $i<j$, there is $k<j$ satisfying that $\sigma_k\cup\sigma_j\subset\sigma_i\cup\sigma_j$ and $|\sigma_k\cup\sigma_j|=|\sigma_j|+1$. Then the ordering of facets $\sigma_1^\vee<\cdots<\sigma_r^\vee$ of $K^\vee$ is a shelling.
\end{lemma}

\begin{proof}
  The assumption is equivalent to that for any $i<j$, there is $k<j$ satisfying that $\sigma_k^\vee\cap\sigma_j^\vee\supset\sigma_i^\vee\cap\sigma_j^\vee$ and $\sigma_k^\vee\cap\sigma_j^\vee$ is $(m-|\sigma_j|-2)$-dimensional. Then we get that for $j\ge 2$, $\langle \sigma_1^\vee,\ldots,\sigma_{j-1}^\vee\rangle\cap\langle\sigma_j^\vee\rangle$ is pure and $(m-|\sigma_j|-2)$-dimensional, completing the proof.
\end{proof}

\begin{proposition}
  Directed MF-complexes are dual shellable.
\end{proposition}

\begin{proof}
  Let $K$ be a directed MF-complex such that an ordering of minimal non-faces $\sigma_1<\cdots<\sigma_r$ gives the directed MF-complex structure of $K$. Let $I$ be the set of all 1-dimensional minimal non-faces of $K$ and put $\{\tau_1,\ldots,\tau_s\}=\{\sigma_1,\ldots,\sigma_r\}-I$, where $\tau_1<\cdots<\tau_s$. We consider the lexicographic ordering on $I$ such that $\{k,l\}<\{k',l'\}\in I$ if $k<k'$ or $k=k',l<l'$. Then the ordering $I<\tau_1<\cdots<\tau_s$ satisfies the condition of Lemma \ref{shelling}, where $I\sqcup\{\tau_1,\ldots,\tau_s\}$ is the set of all minimal non-faces of $K$. Thus the proof is done.
\end{proof}

Summarizing, we have obtained the implications:
$$\text{directed MF}\quad\Rightarrow\quad\text{dual shellable}\quad\Rightarrow\quad\text{dual totally deletable}\quad\Rightarrow\quad\text{totally fillable}$$




\subsection{Homotopy type}

It is observed in \cite{IK3} that if $K$ is fillable, then $|\Sigma K|$ is homotopy equivalent to a wedge of spheres. Here we consider the naturality of this homotopy equivalence which will be used later.

\begin{proposition}
\label{totally fillable}
If $K$ is fillable with a filling $\F(K)$, then there is a homotopy equivalence
$$|\Sigma K|\simeq\bigvee_{\sigma\in\F(K)}S^{|\sigma|-1}$$
such that for a fillable subcomplex $L$ of $K$ with a filling $\F(L)$ satisfying $\F(L)\subset\F(K)\cup K$, the square diagram
$$\xymatrix{|\Sigma L|\ar[r]^(.38)\simeq\ar[d]&\bigvee_{\tau\in\F(L)}S^{|\tau|-1}\ar[d]^{\widetilde{g}}\\
|\Sigma K|\ar[r]^(.38)\simeq&\bigvee_{\sigma\in\F(K)}S^{|\sigma|-1}}$$
commutes, where $\widetilde{g}\vert_{S^{|\tau|-1}}$ is the inclusion for $\tau\in\F(K)\cap\F(L)$ and the constant map otherwise.
\end{proposition}

\begin{proof}
Put $\overline{K}=K\bigcup_{\sigma\in\F(K)}\sigma$. Since $\overline{K}$ is contractible, there is a homotopy equivalence $CK\to\overline{K}$ which restricts to the identity map of $K$. Then we get the desired homotopy equivalence by pinching $K$ to a point. The assumption on $L$ is equivalent to that $\overline{L}$ is a subcomplex of $\overline{K}$, so one gets the commutative square in the statement.
\end{proof}


\subsection{Contraction ordring}

We define a contraction ordering of vertices of a fillable complex. Let $V$ be a finite set and $S\subset V$ be a subset with $|S|\ge 2$. Let $L$ be a simplicial complex on the vertex set $V$ obtained by attaching trees $T_1,\ldots,T_k$ to $\partial\Delta^S$  by their roots, where $\Delta^S$ is the full simplex on the vertex set $S$ and $\partial\Delta^S$ is its boundary. Let $V_i$ be the vertex set of $T_i$ and $r_i\in V_i\cap S$ be the root of $T_i$. Then one has $V=S\sqcup(V_1-r_1)\sqcup\cdots\sqcup(V_k-r_k)$. An ordering $v_1<\cdots<v_n$ of $V_i-r_i$ is called a local contraction ordering if the full subcomplex $(T_i)_{V_i-\{v_j,\ldots,v_n\}}$ is connected for any $j=1,\ldots,n$. An ordering of $V-S$ is called a contraction ordering if it is the join of local contraction orderings of $V_i-r_i$. Note that a deformation retract of $|L|$ onto $|\partial\Delta^S|$ is  given by a contraction ordering.

For a finite set $V$ and its non-empty subset $S$, let $\Delta(V,S)$ be the simplicial complex which is the disjoint union of $\partial\Delta^S$ and vertices $v\in V-S$.

\begin{proposition}
If $K$ is fillable and $\sigma\in\F(K)$, then there are trees $T_1,\ldots,T_k$ such that there is a subcomplex of $\overline{K}$ obtained by attaching $T_1,\ldots,T_k$ to $\partial\Delta^\sigma$ by their roots.
\end{proposition}

\begin{proof}
Choose any maximal tree of $\overline{K}$. Then since $\overline{K}$ is connected, the vertex set of $T$ is $[m]$. If we remove all edges of $\partial\Delta^\sigma$ from $T$, then we get a collection of trees which gives the desired subcomplex.
\end{proof}

Then we can define a contraction ordering of $[m]-\sigma$ for a fillable complex $K$ and $\sigma\in\F(K)$.


\section{Polyhedral products and the map $w$}


\subsection{Polyhedral product}

Let $(\underline{X},\underline{A})=\{(X_i,A_i)\}_{i=1}^m$ be a collection of pairs of spaces. The polyhedral product of $(\underline{X},\underline{A})$ associated with $K$ is defined in \cite{BBCG} as
$$Z(K;(\underline{X},\underline{A}))=\bigcup_{\sigma\in K}(\underline{X},\underline{A})^\sigma\subset\prod_{i=1}^mX_i$$
where $(\underline{X},\underline{A})^\sigma=Y_1\times\cdots\times Y_m$ such that $Y_i=X_i$ for $i\in\sigma$ and $Y_i=A_i$ for $i\not\in\sigma$. The most fundamental property of polyhedral products is the following which we will use implicitly. For $\emptyset\ne I\subset[m]$, let $(\underline{X}_I,\underline{A}_I)=\{(X_i,A_i)\}_{i\in I}$.

\begin{proposition}
For $\emptyset\ne I\subset[m]$, $Z(K_I;(\underline{X}_I,\underline{A}_I))$ is a retract of $Z(K;(\underline{X},\underline{A}))$.
\end{proposition}

If all $(X_i,A_i)$ are $(D^2,S^1)$ (resp. $(\C P^\infty,*)$), the resulting polyhedral product is the moment-angle complex $\Z_K$ (resp. $DJ_K$). Hereafter, let $(\underline{X},*)=\{(X_i,*)\}_{i=1}^m$ be a collection of pointed spaces. We will generalize the map $\widetilde{w}\colon\Z_K\to DJ_K$ by the polyhedral products
$$\Z_K(\underline{X})=Z(K;(C\underline{X},\underline{X}))\quad\text{and}\quad DJ_K(\underline{X})=Z(K;(\underline{X},*))$$
which are generalization of $\Z_K$ and $DJ_K$, respectively, where $(C\underline{X},\underline{X})=\{(CX_i,X_i)\}_{i=1}^m$. Here we remark that the same notation $\Z_K(\underline{X})$ is used in \cite{GT2} to express a different polyhedral product $Z(K;(C\Omega\underline{X},\Omega\underline{X}))$, where $\Omega\underline{X}=\{\Omega X_i\}_{i=1}^m$.


\subsection{Decomposition of the map $\widetilde{w}$}

As in \cite{IK3}, there is a homotopy fibration
\begin{equation}
\label{fibration}
\Z_K(\Omega\underline{X})\xrightarrow{\widetilde{w}}DJ_K(\underline{X})\to\prod_{i=1}^mX_i
\end{equation}
which specializes to the homotopy fibration \eqref{fibration_toric}. We decompose the map $\widetilde{w}$ to clarify the point of our study.

Let $\Omega X_i\to PX_i\to X_i$ be the path-loop fibration. Then for each $i$, there is a pair of fibrations $(PX_i,\Omega X_i)\to(X_i^{[0,1]},PX_i)\to(X_i,X_i)$, where the second map is the evaluation at 1, and this induces a homotopy fibration
\begin{equation}
\label{path loop}
Z(K;(P\underline{X},\Omega\underline{X}))\to Z(K;(\underline{X}^{[0,1]},P\underline{X}))\to\prod_{i=1}^mX_i.
\end{equation}
 The maps $C\Omega X_i\to PX_i$, $(s,l)\mapsto [t\mapsto l((1-s)t)]$ and the evaluations $X_i^{[0,1]}\to X_i$ at 0 induce homotopy equivalences $Z(K;(P\underline{X},\Omega\underline{X}))\simeq\Z_K(\Omega\underline{X})$ and $Z(K;(\underline{X}^{[0,1]},P\underline{X}))\simeq DJ_K(\underline{X})$. Then by applying these homotopy equivalences to \eqref{path loop}, one gets the homotopy fibration \eqref{fibration}. Hence one gets the following. Let $w\colon\Z_K(\underline{X})\to DJ_K(\Sigma\underline{X})$ be the map induced from the pinch maps $CX_i\to\Sigma X_i$, where $\Sigma\underline{X}=\{\Sigma X_i\}_{i=1}^m$.

 \begin{proposition}
 \label{decomp w}
 The map $\widetilde{w}\colon\Z_K(\Omega\underline{X})\to DJ(\underline{X})$ is the composite of maps
 $$\Z_K(\Omega\underline{X})\xrightarrow{w}DJ_K(\Sigma\Omega\underline{X})\to DJ_K(\underline{X})$$
 where the second map is induced from the evaluation maps $\Sigma\Omega X_i\to X_i$.
 \end{proposition}

 Thus we study the map $w$ and apply its properties to understand the map $\widetilde{w}$. By definition, the map $w$ has the following naturality.

\begin{proposition}
\label{w naturality}
For a subcomplex $L$ of $K$ on the same vertex set $[m]$, the following diagram commutes.
$$\xymatrix{\Z_L(\underline{X})\ar[r]^w\ar[d]&DJ_L(\underline{X})\ar[d]\\
\Z_K(\underline{X})\ar[r]^w&DJ_K(\underline{X})}$$
\end{proposition}


\subsection{Higher Whitehead product}

Suppose that $K$ consists only of two vertices, where $m=2$. Then we have $\Z_K(\underline{X})=X_1*X_2$ and $DJ_K(\Sigma\underline{X})=\Sigma X_1\vee\Sigma X_2$ so that the map $w\colon\Z_K(\underline{X})\to DJ_K(\Sigma\underline{X})$ is the Whitehead product of the inclusions $\Sigma X_i\to DJ_K$ for $i=1,2$, where $X*Y$ means the join of spaces $X$ and $Y$.

Suppose next that $K=\partial\Delta^{[m]}$. Then we have $\Z_K(\underline{X})=X_1*\cdots*X_m$ and $DJ_K(\Sigma\underline{X})$ is the fat wedge of $\Sigma X_i$, which is the subspace of $\prod_{i=1}^m\Sigma X_i$ consisting of points $(x_1,\ldots,x_m)$ where at least one $x_i$ is the basepoint. Porter \cite{P} defined the universal higher Whitehead product of the inclusions $a_i\colon\Sigma X_i\to DJ_K(\Sigma\underline{X})$ for $i=1,\ldots,m$ by the map $w\colon\Z_K(\underline{X})\to DJ_K(\Sigma\underline{X})$ in this special case that $K$ is the boundary of $\Delta^{[m]}$.

We finally consider general $K$. Suppose that $\sigma\subset[m]$ is a minimal non-face of $K$. Then there is the inclusion $DJ_{\partial\Delta^\sigma}(\Sigma\underline{X}_\sigma)\to DJ_K(\Sigma\underline{X})$, where $\underline{X}_\sigma=\{X_i\}_{i\in\sigma}$. Let $a_i\colon\Sigma X_i\to DJ_K(\Sigma\underline{X})$ be the inclusion for $i=1,\ldots,m$. Then the higher Whitehead product of the inclusions $a_i$ for $i\in\sigma$ is defined as the composite $\Z_K(\underline{X}_\sigma)\xrightarrow{w}DJ_{\partial\Delta^\sigma}(\Sigma\underline{X}_\sigma)\to DJ_K(\Sigma\underline{X})$, which we denote by $w_\sigma$.


\section{Fat wedge filtration}


\subsection{Definition}

For a collection of pointed spaces $\underline{Y}=\{Y_i\}_{i=1}^m$, let $T^i(\underline{Y})$ be the subspace of $\prod_{i=1}^mY_i$ consisting of points $(y_1,\ldots,y_m)$ such that at least $m-i$ of $y_j$ are the basepoints, where $T^i(\underline{Y})$ are called the generalized fat wedge of $Y_i$. Put $\Z_K^i(\underline{X})=\Z_K(\underline{X})\cap T^i(C\underline{X})$. Then there is a filtration
$$*=\Z_K^0(\underline{X})\subset\Z_K^1(\underline{X})\subset\cdots\subset\Z_K^m(\underline{X})=\Z_K(\underline{X})$$
which we call the fat wedge filtration of $\Z_K(\underline{X})$. The fat wedge filtration of $\Z_K(\underline{X})$ is studied in \cite{IK3} and it is shown that the fat wedge filtration connects the homotopy type of $\Z_K(\underline{X})$ and the combinatorics of a simplicial complex $K$.


\subsection{Cone decomposition}

In \cite{IK3}, it is shown that if all $X_i$ are suspensions, then the fat wedge filtration of $\Z_K(\underline{X})$ behaves so well that it is a cone decomposition with explicitly described attaching maps. We recall this fact here. Let $\RZ_K$ be the polyhedral product $\Z_K(\underline{X})$ such that all $X_i$ are $S^0$, which we call the real moment-angle complex. We first recall from \cite{IK3} properties of the fat wedge filtration of $\RZ_K$. We denote the $i$-th filter of the fat wedge filtration of $\RZ_K$ by $\RZ_K^i$.

\begin{theorem}
\label{FWF RZ}
For any $\emptyset\ne I\subset[m]$, there is a map $\varphi_{K_I}\colon |K_I|\to\RZ_{K_I}^{|I|-1}$ satisfying the following properties:
\begin{enumerate}
\item $\RZ_K^i$ is obtained from $\RZ_K^{i-1}$ by attaching cones by $\varphi_{K_I}$ for $|I|=i$ so that
$$\RZ_K^i=\RZ_K^{i-1}\bigcup_{I\subset[m],\,|I|=i}C|K_I|.$$

\item Let $\widehat{K}_I$ be the simplicial complex obtained from $K_I$ by adding all of its minimal non-faces. Then $\varphi_{K_I}$ factors through the inclusion $|K_I|\to|\widehat{K}_I|$.
\end{enumerate}
\end{theorem}

The fat wedge filtration of $\Z_K(\underline{X})$ is not a cone decomposition in general unlikely to $\RZ_K$ in Theorem \ref{FWF RZ}. However, as mentioned above, it is a cone decomposition whenever all $X_i$ are suspensions. This fact is proved in \cite{IK3} only for the moment-angle complex $\Z_K$, but it is also proved in the general case by the same construction using higher Whitehead product. We demonstrate it here. Define a map $\widetilde{\Phi}\colon I^m\times\prod_{i=1}^mX_i\to\prod_{i=1}^mCX_i$ by $\widetilde{\Phi}(t_1,\ldots,t_m,x_1,\ldots,x_m)=((t_1,x_1),\ldots,(t_m,x_m))$. Then $\widetilde{\Phi}$ restricts to a map $\Phi\colon\RZ_K\times\prod_{i=1}^mX_i\to\Z_K(\underline{X})$ such that
$$\Phi^{-1}(\Z_K^{m-1}(\underline{X}))=(\RZ_K\times T^{m-1}(\underline{X}))\cup(\RZ_K^{m-1}\times\prod_{i=1}^mX_i).$$
If $\underline{X}=\Sigma\underline{Y}$ for $\underline{Y}=\{Y_i\}_{i=1}^m$, then there is the higher Whitehead product $\omega\colon Y^{*[m]}\to T^{m-1}(\underline{X})$, where $Y^{*[m]}=Y_1*\cdots*Y_m$. Now we define the map $\overline{\varphi}_K\colon|K|*Y^{*[m]}\to\Z_K^{m-1}(\underline{X})$ by the composite
\begin{multline*}
|K|*Y^{*[m]}=(C|K|\times Y^{*[m]})\cup(|K|\times C(Y^{*[m]}))\\
\xrightarrow{(C\varphi_K\times\omega)\cup(\varphi_K\times C\omega)}(\RZ_K\times T^{m-1}(\underline{X}))\cup(\RZ_K^{m-1}\times\prod_{i=1}^mX_i)\xrightarrow{\Phi}\Z_K^{m-1}(\underline{X}).
\end{multline*}

\begin{theorem}
\label{FWF Z}
If $\underline{X}=\Sigma\underline{Y}$, then the fat wedge filtration of $\Z_K(\underline{X})$ is a cone decomposition such that
$$\Z_K^i(\underline{X})=\Z_K^{i-1}(\underline{X})\bigcup_{I\subset[m],\,|I|=i}C(|K_I|*Y^{*I})$$
where the attaching maps are $\overline{\varphi}_{K_I}$.
\end{theorem}

It is shown in \cite{IK3} that if $\varphi_{K_I}\simeq*$ for any $I$, then $\overline{\varphi}_{K_I}\simeq*$ for any $I$ as a consequence of a more general result. We will prove this fact by a more direct argument, which enables us to consider the naturality among null homotopies.

\begin{proposition}
\label{null homotopy}
If $\varphi_K$ is null homotopic, then so is $\overline{\varphi}_K$. Moreover, if a null homotopy of $\varphi_K$ restricts to that of $\varphi_L$ for a subcomplex $L\subset K$, then we may choose a null homotopy of $\overline{\varphi}_K$ such that it restricts to that of $\overline{\varphi}_L$.
\end{proposition}

\begin{proof}
Suppose that $\varphi_K\simeq*$ and we fix a null homotopy. Then the map $(C\varphi_K\times\omega)\cup(\varphi_K\times C\omega)$ in the definition of $\overline{\varphi}_K$ is homotopic to the composite
\begin{multline*}
(C|K|\times Y^{*[m]})\cup(|K|\times C(Y^{*[m]}))\to(|\Sigma K|\times Y^{*[m]})\cup(*\times C(Y^{*[m]}))\\
\xrightarrow{(f\times\omega)\cup(*\times C\omega)}(\RZ_K\times T^{m-1}(\underline{X}))\cup(\RZ_K^{m-1}\times\prod_{i=1}^mX_i)
\end{multline*}
for a map $f\colon|\Sigma K|\to\RZ_K$ defined by gluing $C\varphi_K$ and the null homotopy of $\varphi_K$. Then for $\Phi(\RZ_K\vee\prod_{i=1}^mX_i)=*$, the map $\overline{\varphi}_K$ factors through the map $f\wedge\omega\colon|\Sigma K|\wedge Y^{*[m]}\to\RZ_K\wedge T^{m-1}(\underline{X})$. Note that $f\wedge\omega=(f\wedge 1_{Y^{*[m]}})\circ(1_{|\Sigma K|}\wedge\omega)=(f\wedge 1_{Y^{*[m]}})\circ\Sigma(1_{|K|}\wedge\omega)$. For $\Sigma\omega\simeq*$, one gets $\Sigma(1_{|K|}\wedge\omega)\simeq*$ so that $\overline{\varphi}_K\simeq*$ as desired. The naturality of null homotopies is obvious by the above deformation of maps.
\end{proof}


\subsection{Homotopy decomposition}

We apply Theorem \ref{FWF Z} to obtain a homotopy decomposition of $\Z_K(\underline{X})$ together with its naturality. To this end, we will use the following simple lemma, where the proof is easy and omitted.

\begin{lemma}
\label{decomp lem}
If a map $\varphi\colon A\to X$ is null homotopic, then there is a homotopy equivalence
$$\epsilon_\varphi\colon X\vee\Sigma A\xrightarrow{\simeq}X\cup_\varphi CA$$
which is natural with respect to $\varphi$ and its null homotopy.
\end{lemma}

By Theorem \ref{FWF Z}, Proposition \ref{null homotopy} and Lemma \ref{decomp lem}, one gets:

\begin{corollary}
\label{decomp general}
Suppose that $\underline{X}=\Sigma\underline{Y}$. If $\varphi_{K_I}\simeq*$ for any $\emptyset\ne I\subset[m]$, then there is a homotopy equivalence
$$\epsilon_K\colon\Z_K(\underline{X})\xrightarrow{\simeq}\bigvee_{\emptyset\ne I\subset[m]}|\Sigma K_I|\wedge\widehat{X}^I$$
where $\widehat{X}^I=\bigwedge_{i\in I}X_i$. Moreover, if $L$ is a subcomplex of $K$ on the vertex set $[m]$ such that a null homotopy of $\varphi_{K_I}$ restricts to that of $\varphi_{L_I}$ for any $\emptyset\ne I\subset[m]$, up to homotopy, then there is a homotopy commutative diagram
$$\xymatrix{\Z_L(\underline{X})\ar[r]^(.34){\epsilon_L}\ar[d]&\bigvee_{\emptyset\ne I\subset[m]}|\Sigma L_I|\wedge\widehat{X}^I\ar[d]\\
\Z_K(\underline{X})\ar[r]^(.34){\epsilon_K}&\bigvee_{\emptyset\ne I\subset[m]}|\Sigma K_I|\wedge\widehat{X}^I}$$
where the vertical arrows are inclusions.
\end{corollary}


\section{Main theorem}


\subsection{Naturality}

By Theorem \ref{FWF RZ}, we have:

\begin{lemma}
\label{totally fillable trivial}
If $K$ is totally fillable, then $\varphi_{K_I}\simeq*$ for any $\emptyset\ne I\subset[m]$.
\end{lemma}

For a totally fillable complex $K$, we put
$$W_K(\underline{X})=\bigvee_{\emptyset\ne I\subset[m]}\bigvee_{\sigma\in\F(K_I)}\Sigma^{|\sigma|-1}\widehat{X}^I.$$
Then by Proposition \ref{totally fillable}, Corollary \ref{decomp general} and Lemma \ref{totally fillable trivial}, one gets the following homotopy decomposition.

\begin{theorem}
\label{decomp Z general}
If $K$ is totally fillable and $\underline{X}=\Sigma\underline{Y}$, then there is a homotopy equivalence
$$\epsilon_K\colon\Z_K(\underline{X})\xrightarrow{\simeq}W_K(\underline{X}).$$
\end{theorem}

By putting $X_i=S^1$ for all $i$, one immediately gets Theorem \ref{decomp Z}. We consider the naturality of the homotopy equivalence $\epsilon_K$ for special subcomplexes of $K$. The following full subcomplex case is obvious by the construction of $\epsilon_K$.

\begin{corollary}
The homotopy equivalence $\epsilon_K$ retracts onto $\epsilon_{K_I}$ for any $\emptyset\ne I\subset[m]$.
\end{corollary}

We have the following naturality too.

\begin{corollary}
\label{naturality Delta}
Suppose that $K$ is totally fillable and $\underline{X}=\Sigma\underline{Y}$. The homotopy equivalence $\epsilon_K$ satisfies that for $\sigma\in\F(K)$, the square diagram
$$\xymatrix{\Z_{\Delta([m],\sigma)}(\underline{X})\ar[r]\ar[d]^{\epsilon_{\Delta([m],\sigma)}}&\Z_K(\underline{X})\ar[d]^{\epsilon_K}\\
W_{\Delta([m],\sigma)}(\underline{X})\ar[r]^g&W_K(\underline{X})}$$
is homotopy commutative, where $g$ restricts to the identity map of $\Sigma^{|\sigma|-1}\widehat{X}^{[m]}$.
\end{corollary}

\begin{proof}
The null homotopy of $\varphi_{K_I}$ is given by the contraction of $|\overline{K_I}|$ which is homotopic to a contraction of $|\overline{\Delta(I,J)}|$ given by a contraction ordering, where $J=\sigma\cap I$. Then the corollary follows from Corollary \ref{decomp general}.
\end{proof}

For $k<m$, put $\widetilde{Z}(k)=(\Z_{\Delta([m-1],[k])}(\underline{X}_{[m-1]})\times X_m)\cup(*\times CX_m)$ and $\widetilde{Z}^i(k)=\widetilde{Z}(k)\cap T^i(C\underline{X})$. Then the following is clear from the definition.

\begin{proposition}
\label{Z(k)}
If $\underline{X}=\Sigma\underline{Y}$, than for each $\emptyset\ne I\subset[m]$, the map $\overline{\varphi}_{\Delta([m],[k])_I}$ restricts to a map $\widetilde{\varphi}_I\colon|\Delta([m-1],[k])_I|*Y^{*I}\to\widetilde{Z}^{|I|-1}(k)$ such that
$$\widetilde{Z}^i(k)=\widetilde{Z}^{i-1}(k)\bigcup_{I\subset[m],\,|I|=i}C(|\Delta([m-1],[k])_I|*Y^{*I})$$
where the attaching maps are $\widetilde{\varphi}_I$.
\end{proposition}

As mentioned above, $\Delta([m],[k])$ is totally fillable. Put $\F(\Delta([m-1],[k])_I)=\F(\Delta([m],k)_I)$ for any $\emptyset\ne I\subset[m-1]$.  Then any null homotopy of $\overline{\varphi}_{\Delta([m],[k])_I}$ given by a contraction ordering induces a null homotopy of $\widetilde{\varphi}_I$. Put
$$\widetilde{W}(k)=\bigvee_{\emptyset\ne I\subset[m-1]}\bigvee_{\sigma\in\F(\Delta([m],[k])_I)}(\Sigma^{|\sigma|-1}\widehat{X}^I\vee\Sigma^{|\sigma|-1}\widehat{X}^{I\cup\{m\}}).$$
Then by Proposition \ref{Z(k)}, one gets:

\begin{corollary}
\label{Z half}
If $\underline{X}=\Sigma\underline{Y}$, then any null homotopies of $\overline{\varphi}_{\Delta([m],[k])_I}$ given by a contraction ordering induces a homotopy equivalence $\widetilde{\epsilon}\colon\widetilde{Z}(k)\to\widetilde{W}(k)$ satisfying a homotopy commutative diagram
$$\xymatrix{\widetilde{Z}(k)\ar[r]\ar[d]^{\widetilde{\epsilon}}&\Z_{\Delta([m],[k])}(\underline{X})\ar[d]^{\epsilon_{\Delta([m],[k])}}\\
\widetilde{W}(k)\ar[r]&W_{\Delta([m],[k])}(\underline{X})}$$
where the horizontal arrows are inclusions.
\end{corollary}

We next show the naturality of the homotopy equivalence $\widetilde{\epsilon}$. Let $q$ be the composite of maps
\begin{multline*}
\Sigma^{k-1}\widehat{X}^{[m]}\simeq(\Sigma^{k-2}\widehat{X}^{[m-1]})*X_m=(C\Sigma^{k-2}\widehat{X}^{[m-1]}\times X_m)\cup(\Sigma^{k-2}\widehat{X}^{[m-1]}\times CX_m)\\
\xrightarrow{\textrm{proj}}(\Sigma^{k-1}\widehat{X}^{[m-1]}\times X_m)\cup(*\times CX_m).
\end{multline*}

\begin{proposition}
\label{naturality half}
If $\underline{X}=\Sigma\underline{Y}$, then the homotopy equivalence $\widetilde{\epsilon}$ of Corollary \ref{Z half} satisfies a homotopy commutative diagram
$$\xymatrix{\Sigma^{k-1}\widehat{X}^{[m]}\ar[r]^(.3)q\ar[d]&(\Sigma^{k-1}\widehat{X}^{[m-1]}\times X_m)\cup(*\times CX_m)\ar[d]\\
\widetilde{W}(k)\ar[d]^{\tilde{\epsilon}^{-1}}&(W_{\Delta([m-1],[k])}(\underline{X}_{[m-1]})\times X_m)\cup(*\times CX_m)\ar[d]^{\epsilon_{\Delta([m-1],[k])}^{-1}\times 1}\\
\widetilde{Z}(k)\ar@{=}[r]&\widetilde{Z}(k)}$$
where the upper vertical arrows are inclusions.
\end{proposition}

\begin{proof}
For pointed spaces $A,B$, we put $A\rtimes B=(A\times B)/(*\times B)$. Let $p$ be the composite of the projection $Y^{*[m-1]}\rtimes\Sigma Y_m\to Y^{*[m-1]}\wedge\Sigma Y_m$ and the natural homotopy equivalence $Y^{*[m-1]}\wedge\Sigma Y_m\simeq Y^{*[m]}$. Then there is a homotopy commutative diagram
$$\xymatrix{Y^{*[m-1]}\rtimes\Sigma Y_m\ar[r]^(.43){w_1\rtimes 1}\ar[d]^p&T^{m-2}(\underline{X}_{[m-1]})\rtimes\Sigma Y_m\ar[d]^{\textrm{incl}}\\
Y^{*[m]}\ar[r]^{w_2}&T^{m-1}(\underline{X})/X_m}$$
where $w_i$ are the higher Whitehead products. Put $L=\Delta([m-1],[k])$. Then by the definition of $\overline{\varphi}_L$, one gets a homotopy commutative diagram
$$\xymatrix{(|L|*Y^{*[m-1]})\rtimes\Sigma Y_m\ar[r]^{\overline{\varphi}_L\rtimes 1}\ar[d]^p&\Z_L^{m-2}(\underline{X}_{[m-1]})\rtimes X_m\ar[d]\\
|L|*Y^{*[m]}\ar[r]^{r\circ\widetilde{\varphi}}&\widetilde{Z}^{m-1}(k)/CX_m}$$
where $r\colon\widetilde{Z}^{m-1}(k)\to\widetilde{Z}^{m-1}(k)/CX_m$ is the projection which is a homotopy equivalence, and the right vertical arrow is the inclusion. Thus by the definitions of $\widetilde{\epsilon}$ and $\epsilon_L$, one obtains a homotopy commutative diagram
$$\xymatrix{\Sigma^{k-1}\widehat{X}^{[m-1]}\rtimes X_m\ar[r]^(.6)p\ar[d]&\Sigma^{k-1}\widehat{X}^{[m]}\ar[d]\\
W_L(\underline{X}_{[m-1]})\rtimes X_m\ar[d]^{\epsilon_L^{-1}\rtimes 1}&\widetilde{W}(k)\ar[d]^{\tilde{\epsilon}^{-1}}\\
\widetilde{Z}(k)/CX_m\ar[r]^(.55){r^{-1}}&\widetilde{Z(k)}.}$$
Since $p\circ q\simeq 1$, the proof is completed.
\end{proof}


\subsection{Main Theorem}

If $K$ is totally fillable, then we fix a contraction ordering of $I-\sigma$ for any $\sigma\in\F(K_I)$ and $\emptyset\ne I\subset[m]$. Let $a_i\colon X_i\to DJ_K(\underline{X})$ be the inclusion for $i=1,\ldots,m$ as above. Now we state the main theorem.

\begin{theorem}
\label{main}
Suppose that $\underline{X}=\Sigma\underline{Y}$ and $K$ is a totally fillable complex. Then for $\sigma\in\F(K_I)$, the composite
$$\Sigma^{|\sigma|-1}\widehat{X}^I\to W_K(\underline{X})\xrightarrow{\epsilon_K}\Z_K(\underline{X})\xrightarrow{w}DJ_K(\Sigma\underline{X})$$
is the iterated Whitehead product
$$[[\cdots[w_\sigma,a_{i_1}],\cdots],a_{i_k}]$$
up to permutation of the smash factors of $\Sigma^{|\sigma|-1}\widehat{X}^I$, where $i_1<\cdots<i_k$ is a contraction ordering of $I-\sigma$.
\end{theorem}

Let $\widetilde{a}_i\colon S^2\to DJ_K$ be the composite of the inclusion of the bottom cell $S^2\to\C P^\infty$ and the inclusion of the $i$-th factor $\C P^\infty\to DJ_K$, and for $\sigma\subset[m]$, let $\widetilde{w}_\sigma$ be the higher Whitehead product of $\widetilde{a}_i$ for $i\in \sigma$ if it is defined. The following is immediate from Theorem \ref{main} and the naturality of (higher) Whitehead products.

\begin{corollary}
If $K$ is a totally fillable complex, then for $\sigma\in\F(K_I)$, the composite
$$S^{|\sigma|+|I|-1}\to\bigvee_{\emptyset\ne I\subset[m]}\bigvee_{\sigma\in\F(K_I)}S^{|\sigma|+|I|-1}\xrightarrow{\epsilon_K}\Z_K\xrightarrow{\widetilde{w}}DJ_K$$
is the iterated Whitehead product
$$[[\cdots[\widetilde{w}_\sigma,\widetilde{a}_{i_1}],\cdots],\widetilde{a}_{i_k}]$$
up to sign, where $i_1<\cdots<i_k$ is a contraction ordering of $I-\sigma$.
\end{corollary}


\subsection{Proof of Theorem \ref{main}}

\begin{lemma}
\label{main lem}
If $\underline{X}=\Sigma\underline{Y}$ and $k<m$, then the composite
$$\Sigma^{k-1}\widehat{X}^{[m]}\to W_M(\underline{X})\xrightarrow{\epsilon_M^{-1}}\Z_M(\underline{X})\xrightarrow{w}DJ_M(\Sigma\underline{X})$$
is the iterated Whitehead product $[[\cdots[w_{[k]},a_{k+1}],\cdots],a_m]$, where $M=\Delta([m],[k])$.
\end{lemma}

\begin{proof}
Put $L=\Delta([m-1],[k])$. By Proposition \ref{naturality half}, we see that there is a homotopy commutative diagram
$$\xymatrix{\Sigma^{k-1}\widehat{X}^{[m]}\ar@{=}[r]\ar[d]^{\overline{w}}&\Sigma^{k-1}\widehat{X}^{[m]}\ar[d]^q\ar@{=}[r]&\Sigma^{k-1}\widehat{X}^{[m]}\ar[dd]\\
\Sigma^{k-1}\widehat{X}^{[m-1]}\vee\Sigma X_m\ar[d]&(\Sigma^{k-1}\widehat{X}^{[m-1]}\times X_m)\cup(*\times CX_m)\ar[l]_(.58){\textrm{proj}}\ar[d]\\
W_L(\underline{X}_{[m-1]})\vee\Sigma X_m\ar[d]^{\epsilon_L^{-1}\vee 1}&(W_L(\underline{X}_{[m-1]})\times X_m)\cup(*\times CX_m)\ar[l]_(.6){\textrm{proj}}\ar[d]^{\epsilon_L^{-1}\times 1}&\widetilde{W}(k)\ar[d]^{\tilde{\epsilon}^{-1}}\\
\Z_L(\underline{X}_{[m-1]})\vee\Sigma X_m\ar[d]^{w\vee 1}&\widetilde{Z}(k)\ar[l]_(.41){\textrm{proj}}\ar@{=}[r]\ar[d]^w&\widetilde{Z}(k)\ar[d]^w\\
DJ_L(\Sigma\underline{X}_{[m-1]})\vee\Sigma X_m\ar@{=}[r]&DJ_L(\Sigma\underline{X}_{[m-1]})\vee\Sigma X_m\ar@{=}[r]&DJ_L(\Sigma\underline{X}_{[m-1]})\vee\Sigma X_m}$$
where $\overline{w}$ is the Whitehead product of the identity maps of $\Sigma^{k-1}\widehat{X}^{[m-1]}$ and $\Sigma X_m$. On the other hand, by Corollary \ref{naturality Delta} there is a commutative diagram
$$\xymatrix{\widetilde{W}(k)\ar[d]^{\tilde{\epsilon}^{-1}}\ar[r]&W_M(\underline{X})\ar[d]^{\epsilon_M^{-1}}\\
\widetilde{Z}(k)\ar[d]^w\ar[r]&\Z_M(\underline{X})\ar[d]^w\\
DJ_L(\Sigma\underline{X}_{[m-1]})\vee\Sigma X_m\ar@{=}[r]&DJ_M(\Sigma\underline{X}).}$$
Then by juxtaposing the above two diagrams, one gets that the composite in the statement is the Whitehead product of the identity map of $\Sigma X_m$ and $w\circ\epsilon_L$. Thus the proof is completed by induction on $m$.
\end{proof}

\begin{proof}
[Proof of Theorem \ref{main}]
The proof is done by Theorem \ref{decomp Z general}, Corollary \ref{naturality Delta} and Lemma \ref{main lem}, where the induction in the proof of Lemma \ref{main lem} is done by a contraction ordering.
\end{proof}

\end{document}